\documentclass{daj}

\dajAUTHORdetails{%
  title = {The Shadow knows: Empirical Distributions of Minimum Spanning Acycles and Persistence Diagrams of Random Complexes}, 
  author = {Nicolas Fraiman, Sayan Mukherjee, and Gugan Thoppe},
  plaintextauthor = {Nicolas Fraiman, Sayan Mukherjee, and Gugan Thoppe},
  runningtitle = {Empirical Distributions of Minimum Spanning Acycles and Persistence Diagrams}, 
  runningauthor = {Nicolas Fraiman, Sayan Mukherjee, and Gugan Thoppe},
  copyrightauthor = {Nicolas Fraiman, Sayan Mukherjee, and Gugan Thoppe},
  keywords = {minimum spanning tree, minimum spanning acycle, persistence diagram, shadow, empirical measure, random graph, random complex, histogram},
}   

\dajEDITORdetails{%
   year={2023},
   number={2},
   received={13 January 2021},   
   published={3 May 2023},  
   doi={10.19086/da.73323},       
}   

\usepackage{amsmath, amsthm, amssymb}
\usepackage{dsfont}
\usepackage{graphics}
\usepackage{mathrsfs}
\usepackage{subcaption}

\newcommand{\Sh}{\textnormal{Sh}}

\newcommand{\normq}[1]{\left\| #1\right\|_q}

\newcommand{\Erdos}{Erd\H{o}s}
\newcommand{\Renyi}{R\'enyi}

\newcommand{\E}{\mathbf{E}}
\renewcommand{\Pr}{\mathbf{P}}

\newcommand{\Real}{\mathbb{R}}

\newcommand{\Lq}{L^q}

\newcommand{\cF}{\mathcal{F}}

\newcommand{\cK}{\mathcal{K}}
\newcommand{\cL}{\mathcal{L}}

\newcommand{\indc}{\mathds{1}}

\newcommand{\cKn}{\cK_{n}}
\newcommand{\Mn}{M_n}
\newcommand{\cLn}{\cL_n}

\newcommand{\Dn}{D_n}
\newcommand{\mup}{\mu_{n,p}}
\newcommand{\mun}{\mu_{n}}

\newcommand{\tmup}{\tilde{\mu}_{n,p}}
\newcommand{\tmun}{\tilde{\mu}_{n}}

\newcommand{\wu}{w_u}
\newcommand{\wg}{w_g}

\newcommand{\argmin}{{\textnormal{arg}\min}}

\newcommand{\df}[1]{\textnormal{ d} #1}

\newtheorem{theorem}{Theorem}
\newtheorem*{theorem*}{Theorem}
\newtheorem{corollary}[theorem]{Corollary}
\newtheorem{lemma}[theorem]{Lemma}

\newtheorem{proposition}[theorem]{Proposition}
\theoremstyle{remark}
\newtheorem{remark}[theorem]{Remark}
\theoremstyle{definition}
\newtheorem{definition}{Definition}

\begin{document}

\begin{frontmatter}[classification=text]

\title{The Shadow Knows: Empirical Distributions of Minimum Spanning Acycles and Persistence Diagrams of Random Complexes} 

\author[nico]{Nicolas Fraiman\thanks{Supported by UNC's Junior Faculty Award funded by IBM and R.J. Reynolds Industries, and NSF grant DMS-2134107.}}
\author[sayan]{Sayan Mukherjee\thanks{Supported by funding from NSF DEB-1840223, NIH R01 DK116187-01, HFSP RGP0051/2017, NSF DMS 17-13012, and NSF CCF-1934964 and the Alexander
von Humboldt Foundation. High-performance computing is partially supported by grant
2016-IDG-1013 from the North Carolina Biotechnology Center. Sayan Mukherjee would
also like to acknowledge the German Federal Ministry of Education and Research within the
project Competence Center for Scalable Data Analytics and Artificial Intelligence (ScaDS.AI)
Dresden/Leipzig (BMBF 01IS18026B).}}
\author[gugan]{Gugan Thoppe\thanks{Supported partially by IISc Start Up Grants SG/MHRD-19-0054 and SR/MHRD-19-0040, partially by DST SERB's Core Research Grant CRG/2021/008330, and partially by the Pratiksha Trust Young Investigator Award.}}

\begin{abstract}
In 1985, Frieze showed that the expected sum of the edge weights of the minimum spanning tree (MST) in the uniformly weighted graph converges to $\zeta(3)$. Recently, Hino and Kanazawa extended this result to a uniformly weighted simplicial complex, where the role of the MST is played by its higher-dimensional analog---the Minimum Spanning Acycle (MSA). Our work goes beyond and describes the histogram of all the weights in this random MST and random MSA. Specifically, we show that their empirical distributions converge to a measure based on a concept called the shadow. The shadow of a graph is the set of all the missing transitive edges and, for a simplicial complex, it is a related topological generalization. As a corollary, we obtain a similar claim for the death times in the persistence diagram corresponding to the above weighted complex, a result of interest in applied topology.
\end{abstract}
\end{frontmatter}


\section{Introduction}
\label{s:introduction}

Random graphs, as models for binary relations, deeply impact discrete mathematics, computer science, engineering, and statistics. However, modern-day data analysis also involves studying models with 
higher-order relations such as random simplicial complexes \cite{kahle2014topology}. Our contributions here are vital from both these perspectives. Specifically, we consider a mean-field model for random graphs and random simplicial complexes and provide a complete description of the distribution of weights in the Minimum Spanning Tree (MST) and its higher dimensional analog---the Minimum Spanning Acycle (MSA). As a corollary, we also obtain the distribution of the death times in the associated persistence diagram. 

To help judge our contributions, we refer the reader to Robert Adler's 2014 article \cite{adler2014article}. There he summed up the state-of-the-art in random topology as follows: while a lot is known about the asymptotic behaviors of the sums of MST and MSA weights and also their associated death times, \emph{almost nothing} is known about the individual values. 
The work in \cite{skraba2017randomly} addressed this gap partially. There, the distributions of the extremal MSA  weights and the extremal death times were studied. In this work, we go beyond and describe the behavior in the bulk. We emphasize that our result is new, even in the graph case.


\subsection{Overview of Key Contributions}
We now provide a summary of our main result along with its visual illustration. The necessary background is given in parallel. 

A simplicial complex $\cK$ on a vertex set $V$ is a collection of subsets of $V$ that is closed under the subset operation. Trivially, every graph is a simplicial complex. However, the closure is what distinguishes a simplicial complex from a hypergraph, the other standard model for studying higher-order relations. In particular, all simplicial complexes are hypergraphs, but the reverse is not true. We refer to an element of $\cK$ with cardinality $k + 1$ as a $k$-dimensional face or simply a $k$-face. Similarly, the $k$-skeleton of $\cK,$ denoted $\cK^k,$ refers to the subset of faces with dimension less than or equal to $k.$

Let $d \geq 1$ and $(\cKn, \wu)$ be  the weighted simplicial complex, where i.) $\cKn$ is the complete $d$-skeleton on $n$ vertices, i.e., $\cKn$ is the set of all subsets that have cardinality less than or equal to $d + 1;$ and ii.) $\wu$ is a weight function such that $\wu(\sigma)$ is an independent uniform $[0, 1]$ random variable if $\sigma$ is a $d$-face, and zero otherwise. Let $M_n$ be a $d$-dimensional MSA\footnote{We provide the formal definition of an MSA in Section~\ref{sec:resultsterm}, but the reader unfamiliar with this term can set $d = 1$ and replace MSA with a MST to read ahead.} (or $d$-MSA in short) in  $(\cKn, \wu).$ Then our main result can be informally stated as follows (see Section~\ref{sec:resultsterm} for the details). 

\begin{theorem*}{(\textbf{Informal summary of our Main Result})} As $n \to \infty,$ the empirical measure related to $\{n \wu(\sigma): \sigma \in \Mn\}$ converges to a deterministic measure $\mu$ related to the asymptotic shadow density of the $d$-dimensional Linial-Meshulam complex $Y(n, c/n),$ $c \geq 0;$ see \eqref{d:limiting.bulk.measure} for the exact definition of $\mu.$
\end{theorem*}

Intuitively, our result says that the normalized count of the $M_n$ face weights---after being scaled by $n$---that lie in any subset of the real line asymptotically converges to that subset's measure under $\mu.$ The next couple of paragraphs briefly explain what this measure $\mu$ is. 

Recall that the \Erdos-\Renyi\ graph, denoted by $G(n, p),$ is a random graph on $n$ vertices where each edge is present with probability $p$ independently. The $d$-dimensional Linial-Meshulam complex $Y_d(n, p)$ or simply $Y(n, p)$ is an analog of this model in higher dimensions. This random complex has $n$ vertices and the complete $(d - 1)$-skeleton; further, the $d$-faces are included with probability $p$ independently. It is easy to see that $Y_1(n, p)$ and $G(n, p)$ are equivalent. 


The shadow of a graph $G$ is the set of all the missing transitive edges. That is, it includes a vertex pair $\{u, v\},$ if $G$'s edge set excludes such a pair, but contains an alternative path from $u$ to $v$ via other edges. Clearly, each vertex pair in $G$'s shadow has endpoints in any one component of $G.$ A simplicial complex's shadow is a related topological generalization; see Definition~\ref{d:Shadow} for details. 

The shadow of the random complex $Y(n, c/n)$ was investigated in \cite{linial2016phase}, and its cardinality, after a suitable normalization, was shown to converge to a deterministic constant $s(c).$ Now, $\mu$ is that measure whose density is, loosely, one minus the function made up of these constants for different $c$ values.



\begin{figure}[t!]
\begin{subfigure}{.495\textwidth}
  \centering
  \includegraphics[width=\linewidth]{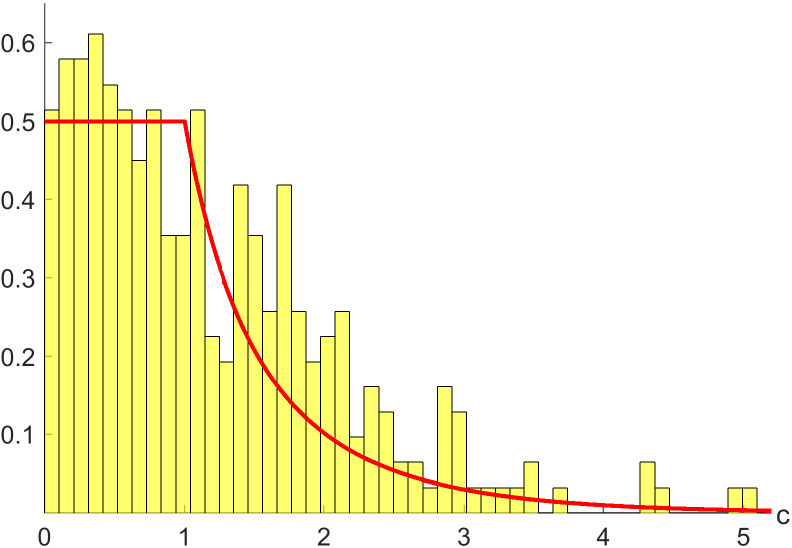}
  \caption{$n = 300, d = 1$}
  \label{fig:Histogram.and.Shadow.d.equal.1}
\end{subfigure}
\begin{subfigure}{.495\textwidth}
  \centering
  \includegraphics[width=\linewidth]{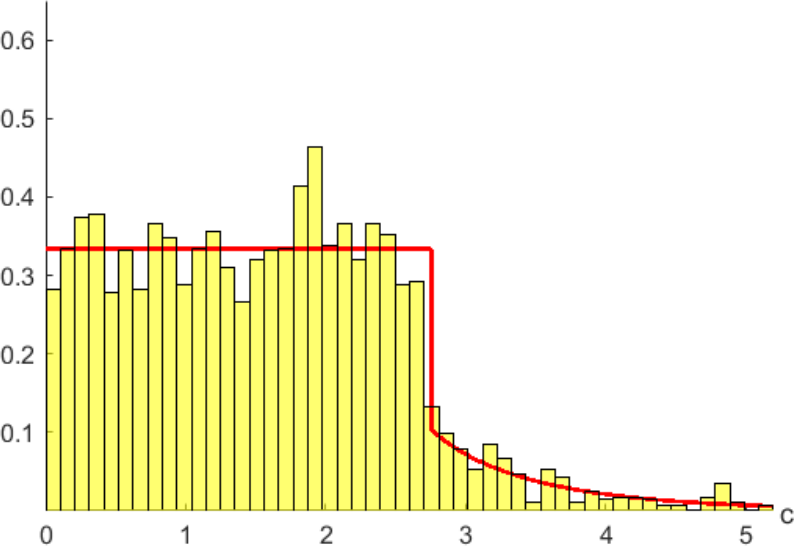}
  \caption{$n = 75, d = 2$}
  \label{fig:Histogram.and.Shadow.d.equal.2}
\end{subfigure}
\caption{Normalized histogram (yellow) of $\{n \wu(\sigma): \sigma \in \Mn\}$ and the density (red) of the shadow based measure $\mu$ given in \eqref{d:limiting.bulk.measure}.
\label{fig:Histogram.and.Shadow}}
\end{figure}

Figure~\ref{fig:Histogram.and.Shadow} illustrates our main result pictorially. The  two scenarios shown there correspond to two different values of the pair $(n, d).$ The yellow plots show the normalized histogram of the set $\{n \wu(\sigma) : \sigma \in \Mn\},$ while the red curves show the density of the shadow based limiting measure $\mu.$  Observe that the red curves and the top of the yellow plots more or less resemble each other. Our main result loosely states that the difference between them decays to zero, as $n \to \infty.$




The weighted simplicial complex $(\cKn, \wu)$ can also be viewed as the evolving simplicial complex $(\cLn(s))_{s \in [0, 1]},$ where $\cLn(s) := \{\sigma \in \cKn: \wu(\sigma) \leq s\}.$ In this evolution, $\cLn(s)$ monotically grows from a simplicial complex with an (almost surely) empty $d$-skeleton (when $s = 0$) to one that has all the $d$-faces (when $s = 1$). There is a natural persistence diagram associated with this process, which records the birth and death times of the different topological holes that appear and disappear as the process evolves. In \cite[Theorem~3]{skraba2017randomly}, it was shown that the set of death times in the $(d - 1)$-th dimension of this diagram exactly equals the set of weights in the $d$-MSA of $(\cKn, \wu).$ Consequently, the above result can also be stated in terms of the death times in the persistence diagram related to $(\cKn, \wu).$ 


\subsection{Related Work}
\label{sec:Related.Work}
Our work lies at the intersection of two broad strands of research: one concerning component sizes, shadow densities, and homologies of random graphs and random complexes, and the other dealing with the statistics of weights in random MSTs and MSAs. In this section, we look at a few of the historical milestones in these two strands.

\Erdos\ and \Renyi\ were the ones who initiated the first strand with their work in \cite{erdos1959random}. There they showed that $p = \log n/n$ is a sharp asymptotic threshold for connectivity in $G(n, p).$ Also that, if $p = (\log n + c)/n$ for $c \in \Real$ and $n \to \infty,$ then i.) almost all the vertices in $G(n, p)$ lie in one single component, and ii.) the vertices outside this  are all isolated and their number has a Poisson distribution with mean $e^{-c}.$ This result was subsequently refined in \cite{Erdos:1960} and the new statement included the following additional facts: i.) the asymptotic order of the largest component jumps from logarithmic to linear around $p = 1/n,$ and ii.) for $c > 1,$ the largest component in $G(n, c/n),$ denoted $L_n(c),$ satisfies  $|L_n(c)|/n \to 1 - t(c),$ where $|S|$ denotes the cardinality of a set $S,$ and  $t(c)$ is the unique root in $(0, 1)$ of the equation
\begin{equation}
\label{e:Giant.Component.Size}
    t = e^{-c(1 - t)}.
\end{equation}

A graph is a one-dimensional simplicial complex, so one can ask if phenomena like the ones above also occur in random complexes. Since  connectivity is related to the vanishing of the zeroth homology, it is natural to look at the higher order Betti numbers  to answer this question. Such a study was done in \cite{linial2006homological, meshulam2009homological}, and it was found that the $(d - 1)$-th Betti number of $Y(n, p)$ indeed shows a non-vanishing to vanishing phase transition at $p = d \log n/n.$ A separate study \cite{kahle2016inside} also showed that this Betti number converges  to a Poisson random variable with mean $e^{-c},$ when $p = (d \log n - \log (d!) + c)/n$ for $c \in \Real.$

The result on component sizes, in contrast, was not so easy to generalize. The challenge was in coming up with a higher dimensional analog of a component in a simplicial complex. The breakthrough came  in \cite{linial2014extremal} with the introduction of the shadow. The underlying  motivation was that, in a sparse graph, a giant component exists if and only if the  shadow includes a positive fraction of all the possible edges. With this in mind, the behavior of $\Sh(Y(n, c/n)),$ the $d$-dimensional shadow (or $d$-shadow) of $Y(n, p),$ was investigated in \cite{linial2016phase}, and these were the key findings there. One, $|\Sh(Y(n, p))|$ changes from $o(n^{d + 1})$ to $\Theta(n^{d + 1})$ at $p = c_*/n,$ where $c_*$ equals $1$ when $d  = 1,$ but is strictly greater than $1$ for $d \geq 2.$ Two, for $c > c_*,$  $|\Sh(Y(n, c/n))| /\binom{n}{d + 1} \to (1 - t(c))^d,$ where $t(c)$ is the smallest root in $(0, 1)$ of the equation
\begin{equation}
\label{eqn:c tc Alt Relation}
    t = e^{-c(1 - t)^d}.
\end{equation}
Note that \eqref{eqn:c tc Alt Relation} matches  \eqref{e:Giant.Component.Size} when $d = 1.$ Also, for $c > 1,$ $|\Sh(Y_1(n, c/n))| = \Theta(n^2),$ while $|L_n(c)| = \Theta(n).$

The pioneering work in the second strand was done  by Frieze \cite{frieze1985value}. He showed that, given a complete graph on $n$ vertices with uniform $[0, 1]$ weights on each edge, the  expected sum of weights in the MST converges to the constant $\zeta(3) \approx 1.2,$ as $n \to \infty.$ Note that this graph is the $d = 1$ case of our $(\cKn, \wu)$ model. Recently, \cite{hino2019asymptotic} showed that a similar result exists even for the $d$-MSA of $(\cKn, \wu)$ for $d \geq 2.$ 

\begin{figure}
    \centering
    \begin{subfigure}{\textwidth}
        \centering
        \includegraphics[width= 0.95\linewidth]{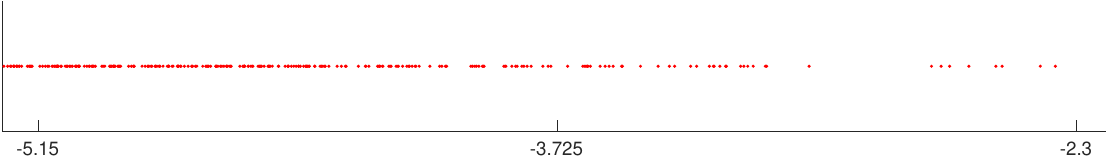}
        \caption{$n = 500, \quad d = 1$}
    \end{subfigure}
\begin{subfigure}{\textwidth}
        \centering
        \includegraphics[width=0.95\linewidth]{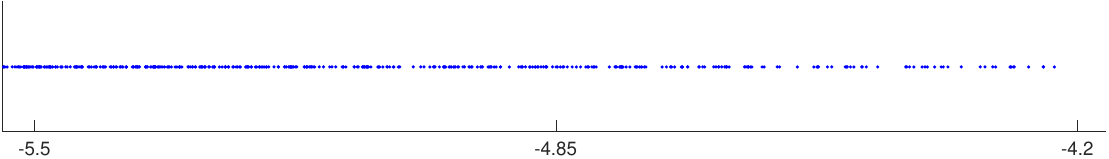}
        \caption{$n = 100, \quad d = 2$}
    \end{subfigure}
\caption{Plot of $\{n \wu(\sigma) - d \log(n) + \log(d!) : \sigma \in \Mn\}.$
    \label{fig:Extremal.Weights}}
\end{figure}

Separately, \cite{skraba2017randomly} studied the extremal weights in the $d$-MSA of $(\cKn, \wu).$  An illustration of a main result there is given in Figure~\ref{fig:Extremal.Weights}. It shows two distinct $(n, d)$ scenarios. In both, a plot of the values in $\{n \wu(\sigma) - d\log (n) + \log(d!): \sigma \in \Mn\}$ is given. This choice of scaling gives more prominence to the extremal weights, which in the two plots are the values on the extreme right. The result in \cite{skraba2017randomly} states that these extremal values converge to a Poisson point process, while the rest go to $-\infty.$

\subsection{Proof Outline}
\label{subsec:proof.overview}
We now briefly describe how we combine the different facts stated above for proving our main result. 
%
%
%
Let $\sigma$ be an arbitrary $d$-face in $(\cKn, \wu)$ and suppose $\wu(\sigma) = s > 0.$ Let $\cLn(s^{-}):= \{\sigma \in \cKn: w(\sigma) < s\}.$ Clearly, the distributions of $\cLn(s^{-})$  and $Y(n, s$), and hence of  their shadows, are identical.  Now, \cite[Lemma~32]{skraba2017randomly} states that $\sigma$ belongs to the $d$-MSA of $(\cKn, \wu)$ if and only if it does not lie in the shadow\footnote{The word `shadow' is not used in \cite{skraba2017randomly}, but the result can be interpreted in terms of the shadow in a straightforward fashion.} of $\cLn(s^{-}).$ Given that $\sigma$ could potentially have been any of the $d$-faces outside of $\cLn(s^{-}),$ the probability that it belongs to the $d$-MSA must then equal one minus the shadow density of $\cLn(s^{-}).$ By shadow density, we mean the shadow's size divided by the total number of potential $d$-faces. Based on this and the shadow density results from \cite{linial2016phase}, the desired result is then easy to see.

\section{Main Result}\label{sec:resultsterm}
We give our main result here along with all the definitions that were skipped in Section~\ref{s:introduction}. All the while, we presume that the reader is well versed with the basics of simplicial homology. If not, the necessary background can be found in \cite[Section 2]{skraba2017randomly} and the references therein. 

For a simplicial complex $\cK,$ we use $\cF^k(\cK)$ and $\beta_k(\cK)$ to denote the set of $k$-faces and the $k$-th reduced Betti number\footnote{Throughout we assume the Betti numbers are defined using real coefficients, as in \cite{linial2016phase}.}, respectively. We drop $\cK$ from these notations, when the underlying simplicial complex is clear. Separately, for a random variable $X,$ we use $\normq{X} = \E(X^q)^{1/q}$ to represent its  $\Lq$-th norm. 




Our first aim is to define spanning acycles \cite{kalai1983enumeration} and MSAs, the main objects of our study. These are topological generalizations of spanning trees and MSTs, respectively. Recall that a spanning tree in a connected graph is a subset of edges that connects all the vertices together without creating any cycles. In that same spirit, a $d$-spanning acycle in a simplicial complex $\cK$ is a subset of $d$-faces which when added to $\cK^{d - 1},$ i.e., the $(d - 1)$-skeleton of $\cK,$ kills the $(d - 1)$-th Betti number of $\cK^{d - 1},$ but does not create any new $d$-cycles. Finally, a MSA in a weighted simplicial complex is simply the spanning acycle with the smallest  possible weight. As shown in \cite[Lemma~23]{skraba2017randomly}, a spanning acycle and, hence, a $d$-MSA exists in $\cK$ if and only if $\beta_{d - 1}(\cK) = 0.$

\begin{definition}[Spanning Acycle]
Let $\cK$ be a simplicial complex with $\beta_{d - 1}(\cK) = 0$. Then, a \emph{$d$-spanning acycle} of $\cK$ is a set $S$ of $d$-faces in $\cK$ such that $\beta_{d - 1}(\cK^{d - 1} \sqcup S) = \beta_d(\cK^{d - 1} \sqcup S) =  0$.
\end{definition}

\begin{definition}[Minimum Spanning Acycle]
Let $(\cK, w)$ be a weighted $d$-complex with $\beta_{d - 1}(\cK) = 0$. Then, a \emph{$d$-MSA} of $(\cK, w)$ is an element of $\argmin_{S} \{w(S)\}$, where the minimum is taken over all the $d$-spanning acycles and $w(S)$ represents the sum of weights of the faces in $S$.
\end{definition}

We remark that an MSA is unique when the $d$-faces in $\cK$ have unique weights, e.g., see \cite{skraba2017randomly}.

Next, we discuss the concept of a shadow. Introduced in \cite{linial2014extremal}, the shadow of a graph is the set of all those missing edges which when added will create a cycle. A simplicial complex's shadow is the topological generalization of this concept. To state the formal definition, we need the notion of the dimension of a simplicial complex. This is the largest $k$ for which $\cF^k \neq \emptyset.$

\begin{definition}[Shadow]
\label{d:Shadow}
For a $d$-dimensional simplicial complex $\cK$ with vertex set $V$ and having the complete $(d - 1)$-skeleton, its $d$-shadow is given by
\[
    \Sh(\cK) = \left\{\sigma \in \tbinom{V}{d + 1} \setminus \cF^d: \beta_d(\cK \sqcup \sigma) = \beta_d(\cK) + 1 \right\},
\]
where $\tbinom{V}{d + 1}$ denotes the set of all $(d + 1)$-sized subsets of $V.$
\end{definition}
%


In analogy with the terms for the spectrum of random matrices, we next define the empirical measures that we study in this work. Let $\Dn$ be the set of death times in the $(d - 1)$-th persistence diagram of $\left(\cLn(s)\right)_{s \in [0, 1]},$ the evolving simplicial complex associated with $(\cKn, \wu).$ Also, let $\delta$ be the Dirac measure.




\begin{definition}[Empirical Measure]
The \emph{empirical measures} corresponding to the face weights in $\Mn$ and the death times in $\Dn$ are the random measures respectively given by
\begin{equation}
    \label{d:empirical.measure}
    \mun := \frac{1}{\binom{n - 1}{d}} \sum_{\sigma \in \Mn} \delta_{n w(\sigma)} \quad \text{ and } \quad \tmun := \frac{1}{\binom{n - 1}{d}} \sum_{\Delta \in \Dn} \delta_{n \Delta}.
\end{equation}
\end{definition}



Finally, we define the deterministic measure $\mu$ which serves as the limit for both $\mun$ and $\tmun.$ Let $t_*$ be the smallest root in $(0, 1]$ of the equation 
\begin{equation}
\label{d:t_*}
    (d + 1)(1 - t) + (1 + dt) \log t = 0.
\end{equation}
It can be checked that $t_* = 1$ for $d = 1$ and is $ < 1$ for $d \geq 2.$ Next, let $c_* = \lim_{t \to t_*^{-}} \psi(t),$ where $\psi(t) := - \log 
 t/(1 - t)^d$ for $t \in (0, 1).$
That is, let
\begin{equation}
\label{d:c_*}
    c_* = 
    \begin{cases}
    1, & d = 1, \\
    -\dfrac{\log t_*}{(1 - t_*)^d}, & d \geq 2.
    \end{cases}
\end{equation}
Now, for $c \geq c_*,$ let $t(c)$ be as in Section~\ref{sec:Related.Work}, i.e., the smallest root in $(0, 1]$ of \eqref{eqn:c tc Alt Relation}. Clearly, $t_* = t(c_*),$ i.e., $t_* = e^{-c_*(1 - t_*)^d},$ for any $d \geq 1.$ 
%
%
Building upon these  terms, let $\mu$ be the measure whose  density is given by
\begin{equation}
\label{d:limiting.bulk.measure}
    \df{\mu(x)} = \frac{1 - s(x)}{d + 1}\df{x}, \quad x \geq 0,
\end{equation}
where
\begin{equation}
    \label{d:s.Val}
    s(x) =
    \begin{cases}
    0, & x \leq c_*, \\
    \big(1 - t(x)\big)^{d + 1}, & x > c_*.
    \end{cases}
\end{equation}
We remark that this $c_*$ is the same constant that we came across in Section~\ref{subsec:proof.overview}. Separately, it follows from \cite[Theorem~1.4]{linial2016phase} that $s(c)$ for $c > c_*$ represents the asymptotic density of the shadow of $Y(n, c/n),$ i.e., 
\[
    s(c) = \lim_{n \to \infty} \frac{|\Sh(Y(n, c/n))|}{\binom{n}{d + 1}}.
\]
%


We now state our main result. Define $K(\rho, \rho')$ to be the Kolmogorov metric between the two measures $\rho$ and $\rho'$, i.e., let
\[
    K(\rho,\rho') = \sup_{x \in \Real} \big|\rho(-\infty,x]  - \rho'(-\infty,x] \big|.
\]




\begin{theorem}[Main Result]\label{thm:bulk}
Let $q \geq 1$. Then the random measures $\mun$ and $\tmun$ converge in the Kolmogorov metric to $\mu$ in $\Lq$. Moreover, if $f: \Real \to \Real$ is a continuous function with $\mu|f| < \infty$ and
\begin{equation}
    \label{eqn:Suff.Cond.Lq}
    \lim_{b \to \infty} \sup_{n \geq 0} \normq{\mun |f| - \mun (|f| \wedge b) \big.} = 0,
\end{equation}
then both $\mun f$ and $\tmun f$ converge to $\mu f$ in $\Lq$.
\end{theorem}

\begin{corollary}
    \label{cor:Spl.Case}
    For the unbounded function $f(x) = x^\alpha,$  $\alpha > 0,$ the hypothesis of Theorem~\ref{thm:bulk} holds true for any $q \geq 1.$ That is, $\mu|f| < \infty$ and \eqref{eqn:Suff.Cond.Lq} holds. 
\end{corollary}


\begin{remark}
\label{r:Lnp.extension.Ynp}
In Section~\ref{sec:extensions}, we discuss three extensions of the above result. The first two relate to the cases where the $d$-face weights have a generic distribution and an additional noisy perturbation. The third one is about the weighted $Y(n, p)$ complex, wherein not all the potential $d$-faces may be present. 
\end{remark}

\begin{remark}
Theorem~\ref{thm:bulk} readily applies when $f$ is bounded since hypothesis \eqref{eqn:Suff.Cond.Lq} is then trivially satisfied. Hence, an immediate consequence of our result is that $\mun$ converges to $\mu$ weakly in $\Lq$. However, our result also applies to unbounded functions such as $f(x) = x^\alpha$ for $\alpha > 0$. This is an important example since Frieze's result \cite{frieze1985value} concerning the sum of weights in the MST and its recent generalization to higher  dimensions by Hino and Kanazawa (\cite[Theorem~4.11]{hino2019asymptotic}) then become special consequences of Theorem~\ref{thm:bulk}. Moreover, the limiting constant $I_{d - 1}^{(\alpha)}$ in \cite{hino2019asymptotic} can now be interpreted as the  $\alpha$-th moment of the measure $\mu.$
\end{remark}

\begin{remark}\label{rmk:inprob}
A variant of the second part of Theorem~\ref{thm:bulk}, obtained by replacing condition \eqref{eqn:Suff.Cond.Lq} by 
\begin{equation}\label{eqn:bapprox.inprob}
\lim_{b \to \infty} \sup_{n \geq 0} \Pr\{\mun |f| - \mun (|f| \wedge b) \geq \epsilon\} = 0    
\end{equation} 
and convergence in $\Lq$ by convergence in probability, also holds. Notice that, while the condition on $f$ that we require here is weaker, the conclusion is also weaker. The details are given in Proposition \ref{prop:inprob}. 
\end{remark}

\section{Proofs}\label{sec:bulk}

We derive here all the results stated in Section~\ref{sec:resultsterm} and, also, Proposition~\ref{prop:inprob} which concretizes the claim in  Remark~\ref{rmk:inprob}. We emphasize once again that both Theorem~\ref{thm:bulk} and Proposition~\ref{prop:inprob} hold for a class of unbounded functions, examples of which are given in Corollary~\ref{cor:Spl.Case}.

We begin with two propositions.
\begin{proposition}
\label{prop:mu.well.defined}
$\mu$ is a probability measure with the density function given in \eqref{d:limiting.bulk.measure}.
\end{proposition}

\begin{proposition}
\label{prop:conv on Intervals}
Let $c \geq 0$ and $q \geq 1.$ Then, $\lim_{n \to \infty} \normq{\mun(c, \infty) - \mu(c, \infty)} = 0.$
\end{proposition}

The latter proposition concerns a particular case of Theorem~\ref{thm:bulk}, where $f = \indc(c, \infty)$ for some $c \geq 0.$ Notice that Proposition~\ref{prop:mu.well.defined} and the boundedness of $f$ imply that the condition in \eqref{eqn:Suff.Cond.Lq} and $\mu |f| < \infty$ trivially hold.  
The proofs of these two results are postponed to the end of this section.




\begin{proof}[Proof of Theorem~\ref{thm:bulk}]
Due to  \cite[Theorem~3]{skraba2017randomly}, it suffices to derive the results only for  $\mun.$ We first show that $K(\mun,\mu) \to 0$ in $\Lq$. Let $G_{n}(x) = \mun(-\infty,x]$ and $G(x) = \mu(-\infty,x]$.
Fix $m$ and pick $c_0, c_1, \ldots, c_m$ such that $c_0 = 0$, $c_m = \infty,$ and $G(c_{i + 1}) - G(c_i) = 1/m$, which can be done due to Proposition~\ref{prop:mu.well.defined}. 

For $x \in [c_i, c_{i + 1}]$, we get
\begin{align*}
    |G_{n}(x) - G(x)| \leq {} &\max\big\{ G(c_{i + 1}) - G_{n}(c_i),\; G_{n}(c_{i + 1}) - G(c_i)\big\} \\
    \leq {} & \max\left\{G(c_i) + \frac{1}{m} - G_{n}(c_i),\; G_{n}(c_{i + 1}) - G(c_{i + 1}) + \frac{1}{m} \right\} \\
    \leq {} & \sum_{j=1}^m|G_{n}(c_j) - G(c_j)| + \frac{1}{m}.
\end{align*}
Since the bound is independent of $x,$ it then follows that
\begin{equation}
\label{e:Kolmogorov.Relation}
K(\mun, \mu) = \sup_{x\in\Real} |G_{n}(x) - G(x)| \leq \sum_{j=1}^m|G_{n}(c_j) - G(c_j)| + \frac{1}{m}.
\end{equation}
Therefore, 
\begin{align*}
\normq{K(\mun,\mu)} 
%
\leq {} & \sum_{j=1}^m \normq{G_{n}(c_j) - G(c_j)} + \frac{1}{m}.
\end{align*}
Clearly, $G(x)=1 - \mu(x,\infty)$ and $G_{n}(x) = 1 - \mun(x,\infty).$ Thus, by Proposition \ref{prop:conv on Intervals},  $\normq{G_{n}(c_j) - G(c_j)} \to 0$ for all $j.$ Now, since $m$ is arbitrary, $\normq{K(\mun,\mu)} \to 0,$ as desired.\medskip

To extend the convergence to unbounded functions satisfying the given conditions, we begin by assuming that $f$ is non-negative. Let $\epsilon > 0$ be arbitrary. For any $b > 0$, triangle inequality shows
\[
    \normq{\mun f - \mu f} \leq \normq{\mun f - \mun (f \wedge b)} + \normq{\mun (f \wedge b) - \mu (f \wedge b)} + \normq{\mu f - \mu (f \wedge b)}.
\]
Now pick a large enough $b > 0$ so that
\[
    \sup_{n \geq 0} \big\|\mun f - \mun (f \wedge b)\big\|_q \leq \epsilon \quad \text{ and } \quad \normq{\mu f - \mu (f \wedge b)} = \mu f - \mu (f \wedge b) \leq \epsilon.
\]
Such a $b$ indeed exists due to \eqref{eqn:Suff.Cond.Lq} and the fact that $\mu f < \infty.$ Therefore,
\[
 \normq{\mun f - \mu f} \leq 2\epsilon + \normq{\mun (f \wedge b) - \mu (f \wedge b)}.
\]
However, $f \wedge b$ is a bounded continuous function. Also, the Kolmogorov metric dominates the L\'evy metric which metrizes weak convergence. Consequently, $\limsup_{n \to \infty} \normq{\mun f - \mu f} \leq 2 \epsilon. $ Since $\epsilon > 0$ is arbitrary, it follows that $\mun f$ converges $\mu f$ in $\Lq$.

It remains to deal with the case of general $f$. Clearly, $f = f^+ - f^-$, where $f^+ = \max\{f, 0\}$ and $f^- = \max\{-f, 0\}$. Furthermore, both $(f^+ - b) \; \indc[f^+ \geq b]$ and $(f^- - b) \; \indc[f^- \geq b]$  are bounded from above by $  (|f| - b)\; \indc[|f| \geq b]$, whence it follows that $f^+ - (f^+ \wedge b)$ and $f^- - (f^- \wedge b))$ are bounded from above by $|f| - |f| \wedge b$. Therefore, by repeating the above arguments for both $f^+$ and $f^-$, individually, it is easy to see that the result holds for the case of general $f$ as well.
\end{proof}

We next discuss the claim in Remark \ref{rmk:inprob}. Formally, it can be stated as follows. 

\begin{proposition}\label{prop:inprob}
$K(\mun, \mu) \to 0$ in probability. Moreover, if $f: \Real \to \Real$ is a continuous function such that \eqref{eqn:bapprox.inprob} holds for every $\epsilon > 0,$ then $\mun f$ converges to $\mu f$ in probability.   
\end{proposition}

\begin{proof}
From  Proposition~\ref{prop:conv on Intervals} and Markov's inequality, we have that $\mun (c, \infty)$ converges to $\mu(c, \infty)$ in probability for any $c \geq 0.$ This along with \eqref{e:Kolmogorov.Relation} then shows that $K(\mun, \mu) \to 0$ in probability, as desired. The fact that the Kolmogorov metric dominates the L\'evy metric then immediately shows that $\mun f$ converges to $\mu f$ in probability whenever $f: \Real \to \Real$ is a bounded continuous function. 

We now discuss the case of unbounded functions. Again, as in the proof of Theorem~\ref{thm:bulk}, it suffices to derive the result assuming $f$ to be  non-negative. Let $\epsilon, \eta > 0$. Now pick a $b > 0$ so that
\[
    \sup_{n \geq 0} \Pr\{|\mun f - \mu(f \wedge b)| > \epsilon\} \leq \eta \quad \text{ and } \quad |\mu f - \mu (f \wedge b)| \leq \epsilon.
\]
This can be done on account of \eqref{eqn:bapprox.inprob} and the fact that $\mu f < \infty.$ We then have
\[
    \{|\mun f - \mu f| > 3\epsilon\} \subseteq \{|\mun f - \mun (f \wedge b)| > \epsilon\} \cup \{|\mun (f \wedge b) - \mu (f \wedge b)| > \epsilon\},
\]
whence it follows that
\[
    \Pr\{|\mun f - \mu f| > 3\epsilon\} \leq \eta + \Pr\{|\mun (f \wedge b) - \mu (f \wedge b)| > \epsilon\}.
\]
However, $f \wedge b$ is bounded and continuous and, hence, $\mun (f \wedge b)$ converges to $\mu (f \wedge b)$ in probability as shown above. Consequently, we have that  $\limsup_{n \to \infty} \Pr\{|\mun f - \mu f| > 3\epsilon\} \leq \eta.$ Now, since $\epsilon, \eta$ are arbitrary, the desired result is easy to see. \end{proof}

We now exploit a recent bound on Betti numbers  \cite{hino2019asymptotic} to prove Corollary~\ref{cor:Spl.Case} and thereby show that unbounded functions such as polynomials indeed satisfy the hypothesis of Theorem~\ref{thm:bulk}. Since \eqref{eqn:bapprox.inprob} is implied by \eqref{eqn:Suff.Cond.Lq}, such functions also satisfy the assumptions of Proposition~\ref{prop:inprob}. A technical result is also needed for deriving Corollary~\ref{cor:Spl.Case}. We state it here, but prove it afterwards.

\begin{lemma}
\label{lem:c tc Limiting behavior}
$t(c)$ is continuous over $(c_*, \infty)$ and  $\lim_{c \to c_*^+}t(c) = t_*.$ Moreover, $t(c) = O(e^{-c/2^d})$ as $c \to \infty.$
\end{lemma}

\begin{proof}[Proof of Corollary~\ref{cor:Spl.Case}] 
For any $\alpha > 0,$ Lemma~\ref{lem:c tc Limiting behavior}, along with  the fact that $0 \leq t(x) \leq 1$ for
$x \geq c_*,$ shows that $x^\alpha [1 - (1 - t(x))^{d + 1}] = O(x^\alpha t(x)) = O(x^\alpha e^{-x/2^d})$ as $x \to \infty.$ It is then straightforward to see from \eqref{d:limiting.bulk.measure} that $\mu |f| < \infty$ for $f(x) = x^\alpha,$ as desired. 

We next show that $f(x) = x^\alpha$ satisfies condition \eqref{eqn:Suff.Cond.Lq} as well. Recall that $\beta_{d - 1}(\cK)$ is the $(d - 1)$-th Betti number of the simplicial complex $\cK.$ Since the set of $d$-MSA weights in a weighted\footnote{This result requires that the weights be monotone, i.e., $w(\sigma) \leq w(\sigma')$ whenever $\sigma \subseteq \sigma'.$} complex equals the set of death times in the associated $(d - 1)$-th persistence diagram \cite[Theorem~3]{skraba2017randomly}, it follows by arguing as in the proof of \cite[Proposition~4.9]{hino2019asymptotic} that
\begin{align*}
\mun f - \mun (f \wedge b)
    &= \frac{n^\alpha }{\binom{n-1}{d}} \sum_{\sigma \in \Mn} \left[ w(\sigma)^\alpha -  w(\sigma)^\alpha \wedge \frac{b}{n^\alpha }\right] \\
    &= \frac{n^{\alpha} \alpha}{\binom{n-1}{d}} \int_{b^{1/\alpha}/n}^{1}\beta_{d - 1}(\cLn(s)) s^{\alpha - 1} \df{s}.
\end{align*}
Now substituting $r = s n$ shows that
\begin{align*}
    \mun f - \mun (f \wedge b) 
    &= \frac{\alpha}{\binom{n-1}{d}} \int_{b^{1/\alpha}}^{n}  r^{\alpha - 1}\beta_{d - 1}(\cLn(r/n)) \df{r}\\
    &= \frac{\alpha n^d}{\binom{n-1}{d}} \int_{b^{1/\alpha}}^{n}  r^{\alpha - 1} \frac{\beta_{d - 1}(\cLn(r/n))}{n^d} \df{r}.
\end{align*}
By using Minkowski's inequality applied to integrals, we then get
\[
    \normq{\mun f - \mun (f \wedge b)} \leq  \frac{\alpha n^d}{\binom{n-1}{d}} \int_{b^{1/\alpha}}^{n} r^{\alpha - 1} \normq{\dfrac{\beta_{d - 1}(\cLn(r/n))}{n^{d}}} \df{r}.
\]
Finally, since $\cLn(s)$ has the same distribution as that of $Y(n, s)$ for $s \in [0, 1],$ it follows that
\begin{equation}
\label{eq:minkowski}
    \normq{\mun f - \mun (f \wedge b)} \leq  \frac{\alpha n^d}{\binom{n-1}{d}} \int_{b^{1/\alpha}}^{n} r^{\alpha - 1} \normq{\dfrac{\beta_{d - 1}(Y(n, r/n))}{n^{d}}} \df{r}.
\end{equation}

Now, $0 \leq \beta_{d - 1}(Y(n, r/n)) \leq \binom{n}{d},$ whence
\[
0\leq \dfrac{\beta_{d - 1}(Y(n, r/n))}{n^{d}} \leq \frac{1}{d!}.
\]
Moreover, (4.14) from \cite{hino2019asymptotic} shows that there is a $\ell > 1 \vee q\alpha$ such that
\[
    \E\left[\frac{\beta_{d - 1}(Y(n, r/n))}{n^{d}}\right] \leq 1 \wedge \dfrac{C}{r^{\ell}}.
\]
By writing $q = q - 1 + 1$ and then using the above two inequalities, we get
\[
    \normq{\dfrac{\beta_{d - 1}(Y(n, r/n))}{n^{d}} } \leq \left(\frac{1}{d!}\right)^{(q-1)/q}\left(1 \wedge \dfrac{C^{1/q}}{r^{\ell/q}} \right)
\]
for all $r \in [0, n]$. Now, applying this bound in \eqref{eq:minkowski}, it follows that there exists a constant $K \geq 0$ such that  
\[
    \normq{\mun f - \mun(f \wedge b)} \leq \frac{K}{\ell/q - \alpha} b^{-(\ell/q - \alpha)/\alpha}
\]
for all sufficiently large $b$ (so that $C^{1/q}/r^{\ell/q} \leq 1$ for all $r \geq b^{1/\alpha}$). This then shows that the condition in \eqref{eqn:Suff.Cond.Lq} holds, as desired.
\end{proof}

It remains to prove Lemma~\ref{lem:c tc Limiting behavior}, and Propositions~\ref{prop:mu.well.defined} and \ref{prop:conv on Intervals}. 
%

\begin{proof}[Proof of Lemma~\ref{lem:c tc Limiting behavior}]
In \cite[Appendix B]{linial2016phase}, it is shown that $\psi(t)$ (defined above \eqref{d:c_*}) is continuous and monotonically decreasing for $0 < t < t(c_*).$ The monotonicity implies that i.) $t(c) = \psi^{-1}(c)$ is continuous over $(c_*, \infty),$ ii.) $t_* = t(c_*) = \lim_{c \to c_*^+} t(c),$ and iii.)  $\lim_{c \to \infty} t(c) = 0.$ The latter fact in turn shows $1 - t(c) \geq 1/2$ (say) whenever $c$ is large. This, combined with  \eqref{eqn:c tc Alt Relation}, implies $t(c) = O( e^{-c/2^d}),$ as desired.
\end{proof}

To prove Propositions~\ref{prop:mu.well.defined} and~\ref{prop:conv on Intervals}, we need two additional technical results. Let
\begin{equation}
\label{d:func.g}
    g(c) =
    \begin{cases}
        0,  & 0 \leq c \leq c_*, \\
        \displaystyle ct(c)\big(1 - t(c)\big)^d + \frac{c}{d + 1}\big(1 - t(c)\big)^{d + 1} - \big(1 - t(c)\big), & c > c_*,
    \end{cases}
\end{equation}
and $h(c) = g(c) + 1 - c/(d + 1).$

\begin{lemma}
\label{lem:Beh d - 1 Betti Number}
For $c \geq 0,$
\[
    \lim_{n \to \infty} \normq{\frac{\beta_{d - 1}(Y(n, c/n))}{n^d} \indc_{[0, n]}(c)  - \frac{h(c)}{d!}} = 0.
\]
\end{lemma}
\begin{proof}
This is shown in the proof of \cite[Theorem~4.11]{hino2019asymptotic}.
\end{proof}

\begin{lemma}
\label{lem:h is in fact int of 1 - shadow density}
$\mu$ is a well-defined measure. Moreover, for any $c \geq 0$, we have $\mu(c,\infty) = h(c)$.
\end{lemma}
\begin{proof}
Lemma~\ref{lem:c tc Limiting behavior} shows  that the density function $s(x)$ is continuous everywhere for $d = 1,$ and discontinous only at $c_*$ for $d \geq 2.$ This implies that the measure $\mu$ is well-defined.

With regards to the other claim, we first show that  $\mu(c,\infty)$ is finite for all $c \geq 0.$  Clearly,
\[
    \mu(c,\infty) \leq \frac{2}{d+1}\max\left\{\int_0^{c_*} (1 - s(x)) \df{x}, \int_{c_*}^{\infty} (1 - s(x)) \df{x}\right\}.
\]
Separately, since $0 \leq t(x) \leq 1$ for all $x \geq c_*$, it follows from \eqref{d:s.Val} and Lemma~\ref{lem:c tc Limiting behavior} that $1 - s(x) = O(t(x)) = O(e^{-x/2^d}).$ From these observations, it is then easy to see that $\mu(c,\infty)$ is finite, as desired.

Next, we claim that $h$ is continuous over $[0, \infty).$ For $c \neq c_*,$ this follows from $h$'s definition and  Lemma~\ref{lem:c tc Limiting behavior}. At $c_*,$ this holds due to  \eqref{d:t_*} and \eqref{d:c_*} and since $\lim_{c \to c_*^+} t(c) =  t_*,$ which imply $\lim_{c \to c_*^{+}} g(c) = 0.$


Our final claim is that $\lim_{c \to \infty} h(c) = \mu(c, \infty) = 0$ and $h'(c) = \frac{\df}{\df c}\mu(c, \infty)$ for all $c \neq c_*.$ This and the continuity of $h(c)$ and $\mu(c, \infty)$ at $c_*$ will then show that  $h(\cdot) = \mu(\cdot, \infty),$ as desired. 

Since $\mu(c, \infty), c \geq 0,$ is finite, we have $\lim_{c \to \infty} \mu(c, \infty) = 0.$ In contrast,  $t(c) \in (0, 1]$ implies $h(c) = O(ct(c))$ as $c \to \infty;$ this when combined with  Lemma~\ref{lem:c tc Limiting behavior} then shows that $\lim_{c \to \infty} h(c) = 0.$     

Regarding the statement on the derivatives, we begin by showing that the two functions are differentiable for all $c \neq c_*.$ This statement holds for $h$ on account of \eqref{d:func.g} and \cite[Claim~5.3]{linial2016phase}, while it is true for $\mu(\cdot, \infty)$ simply by definition. From \cite[Claim~5.3]{linial2016phase}, we also have that
\[
    h'(c) =
    \begin{cases}
        - \frac{1}{d + 1}, & c < c_*, \\
        - \frac{1}{d + 1}\big[1 - \big(1 - t(c)\big)^{d + 1}\big], & c > c_*.
    \end{cases}
\]
From this, it is then easy to see that $h'(c) = \frac{\df}{\df c}\, \mu(c,\infty)$ for all $c \neq c_*$, as desired.
\end{proof}

\begin{proof}[Proof of Proposition~\ref{prop:mu.well.defined}]
This follows from Lemma~\ref{lem:h is in fact int of 1 - shadow density} which shows that $\mu(0, \infty) = h(0) = 1.$
\end{proof}

\begin{proof}[Proof of Proposition~\ref{prop:conv on Intervals}]
Let $\{D_i\}$ be the set of death times in the $(d - 1)$-th persistence diagram of the filtration $\left(\cLn(s)\right)_{s \in [0, 1]}.$ Then, for all $n \geq c,$ we have
\begin{align}
\binom{n - 1}{d} \mun (c, \infty) = \sum_{\sigma \in \Mn}\delta_{n w(\sigma)}(c,\infty)
= {} &  \left|\left\{\sigma \in \Mn : w(\sigma) > \frac{c}{ n}\right\}\right| \nonumber \\
= {} & \left|\left\{i : D_i > \frac{c}{ n}\right\}\right| \label{e:conv.death.times} \\
= {} & \beta_{d - 1}\left(\cLn\left(\frac{c}{n}\right) \right) \nonumber \\
\overset{d}{=} {} & \beta_{d - 1}\left(Y\left(n,\frac{c}{n}\right)\right), \nonumber
\end{align}
where \eqref{e:conv.death.times} follows due to \cite[Theorem~3]{skraba2017randomly}, while the last relation holds since $\cLn\left(c/n \right)$ has the same distribution as $Y(n, c/n).$ From Lemma~\ref{lem:Beh d - 1 Betti Number} and the fact that ${n-1 \choose d}/n^d \to 1/d!$, it then follows that $\lim_{n \to \infty} \normq{\mun (c,\infty) - h(c)} = 0.$ The desired result now holds due to  Lemma \ref{lem:h is in fact int of 1 - shadow density}.
\end{proof}

\section{Extensions}
\label{sec:extensions}
We discuss three different extensions of Theorem~\ref{thm:bulk} here. The first one relates to the case where the $d$-face weights come from some generic distribution. The next one concerns the robustness of our result to noisy perturbations in the $d$-face weights. The third and final one is about the randomly weighted $Y(n, p)$ complex, wherein the set of $d$-faces is random and may not include all the potential ones.

\subsection{Generic distribution for $d$-face weights}
\label{subsec:Generic.Weights}

Let $(\cKn, \wg)$ be the weighted simplicial complex, where $\cKn$ is as before, while the weight function $\wg$ is such that i.) $\{\wg(\sigma): \sigma \in \cF^d\}$ are real-valued i.i.d. random variables with some generic distribution $F,$ and ii.) $\wg(\sigma) = 
 -\infty$ whenever $|\sigma| \leq d.$ 
%
%
We claim that a version of our result also holds in this setup. Let

%
\begin{equation}
    \label{d:empirical.measure.generic}
    \mun^g = \frac{1}{\binom{n - 1}{d}} \sum_{\sigma \in \Mn^g} \delta_{n F(\wg(\sigma))} \quad \text{ and } \quad \tmun^g := \frac{1}{\binom{n - 1}{d}} \sum_{\Delta \in \Dn^g} \delta_{n F(\Delta)},
\end{equation}
where $\Mn^g$ is the $d$-MSA in $(\cKn, \wg),$ and $\Dn^g$ is the set of death times in the associated  persistence diagram. 

\begin{corollary}
    \label{cor:generic.reults}
    Suppose $F$ is continuous. Then,
    Theorem~\ref{thm:bulk}  holds for $\mun^g$ and $\tmun^g.$ 
\end{corollary}

\begin{proof}
From $(\cK_n, \wg),$ we construct a new weighted complex $(\cKn, w_F),$ where $w_{F}(\sigma) = F(\wg(\sigma))$ for all $\sigma \in \cKn.$ Since $F$ is continuous,  $\{F(\wg(\sigma)): \sigma \in \cF^d\}$ is a set of i.i.d. $U[0, 1]$ random variables. Also, $F(\wg(\sigma)) = 0,$ whenever $|\sigma| \leq d.$ Hence, Theorem~\ref{thm:bulk}  readily apply to $(\cKn, w_F).$ Furthermore, when viewed in the context of $(\cKn, w_F),$ $\mun^g$ and $\tmun^g$ resemble the measures given in \eqref{d:empirical.measure}. The only thing that remains to be checked is if $
\Mn^g$ is a $d$-MSA in $(\cKn, w_F).$ However, since any distribution function is non-decreasing, this is trivially true. The desired claim is now easy to see. 
\end{proof}

\subsection{Noisy perturbations in $d$-face weights} 
We establish here the robustness of our result to additional noisy perturbations in the $d$-face weights. 

Consider the weighted complex $(\cKn, \wg'),$ where $\cKn$ is as before, and
\[
    \wg'(\sigma) = 
    \begin{cases}
        -\infty & \text{ if $|\sigma| \leq d,$} \\
        \wg(\sigma) + \epsilon_n(\sigma) & \text{ if $\sigma \in \cF^d.$} 
    \end{cases}
\]
In this sum, $\{\wg(\sigma): \sigma \in \cF^d\}$ are real-value i.i.d. random variables with some generic distribution $F,$ while $\{\epsilon_n(\sigma): \sigma \in \cF^d\}$ are a separate set of random variables denoting noisy perturbations in the $d$-face weights. Note that  these latter variables need not be independent nor identically distributed. Let
\[
    \mun' = \frac{1}{\binom{n - 1}{d}} \sum_{\sigma \in \Mn'} \delta_{n F(\wg'(\sigma))} \quad \text{ and } \quad \tmun' := \frac{1}{\binom{n - 1}{d}} \sum_{\Delta \in \Dn'} \delta_{n F(\Delta)},
\]
where $\Mn'$ is the $d$-MSA in $(\cKn, \wg'),$ and $D_n'$ is the set of death times in the associated  persistence diagram.

\begin{corollary}
\label{cor:Noisy.perturbations}
Suppose $F$ is Lipschitz continuous with Lipschitz constant $\zeta > 0,$ and $n \|\epsilon_n\|_\infty \to 0$ in probability, where $\|\epsilon_n\|_\infty := \max |\epsilon_n(\sigma)|.$ Then,  Theorem~\ref{thm:bulk} holds for $\mun'$ and $\tmun'.$ 
\end{corollary}
%

        
%
\begin{proof}
%
%
We only discuss the $\mun'$ result since the  $\tmun'$ case can be dealt with similarly. For the $\mun'$ case, it suffices to show that an analog of Proposition~\ref{prop:conv on Intervals} holds. This is because the arguments from the proof of Theorem~\ref{thm:bulk} can then be again used to get  the actual result.

Let $(\cKn, \wg)$ be the weighted complex defined in Section~\ref{subsec:Generic.Weights}, but this time we couple $\wg$ to the one in the definition of $\wg'.$ Also, let $\Mn^g$ denote the $d$-MSA in $(\cKn, \wg),$ and let $\mun^g$ be as in \eqref{d:empirical.measure.generic}. 

From\footnote{There is a typo in the statement of \cite[Lemma~38]{skraba2017randomly}. In the second and third displays, $\gamma(D_i)$ should be $\gamma(D_i')$ and $\gamma(\phi(\sigma_i))$ must be $\gamma(\phi'(\sigma'_i   )).$ This follows from Theorem~4 in ibid.}  \cite[Lemma~38]{skraba2017randomly}, we have
\[
    \inf_\gamma \max_{\sigma \in \cF^d} |\wg(\sigma) - \gamma(\wg(\sigma))| \leq \|\epsilon_n\|_\infty,
\]
where the infimum is over all the  bijections $\gamma: \{\wg(\sigma): \sigma \in \Mn^g\} \to \{\wg'(\sigma'): \sigma' \in \Mn'\}.$ Hence, for any measurable set $K \subseteq \Real,$ it follows that
\begin{equation}
    \label{e:subsetRelation}
    \mun^g(K^{-\xi_n}) \leq \mun'(K) \leq \mun^g(K^{\xi_n}),
\end{equation}
where $\xi_n := \zeta n \|\epsilon_n\|_\infty,$ and, for $\xi > 0,$
\[
    K^{\xi} := \{x \in \Real: |x - y| \leq \xi \text{ for some $y \in K$ } \} \quad \text{ and } \quad K^{-\xi} := \{x \in K: |x - y| > \xi  \text{ for all $y \not\in K$ }\}.
\]

Let $\xi > 0$ be arbitrary. Then,  
\begin{align}
    |\mun'(c, \infty) - \mun^g & (c, \infty)| \nonumber \\ 
    \leq {} & \max\left\{|\mun^g(c - \xi_n, \infty) - \mun^g(c, \infty)|, |\mun^g(c, \infty) - \mun^g(c + \xi_n, \infty)|\right\} \label{e:Approx.mup'.by.mup} \\
    \leq {} & \mun^g(c - \xi_n, c + \xi_n] \nonumber \\
    \leq {} & \mun^g(c - \xi, c + \xi] \indc[\xi_n \leq \xi] + \mun^g(c - \xi_n, c + \xi_n] \indc[\xi_n > \xi] \nonumber\\
    \leq {} & \mun^g(c - \xi, c + \xi] + \indc[\xi_n > \xi] \label{e:.mup.prob.measure} \\
    \leq {} & |\mun^g(c - \xi, c + \xi] - \mu(c - \xi, c+ \xi]| + \mu(c - \xi, c+ \xi] + \indc[\xi_n > \xi] \nonumber.
\end{align}
where \eqref{e:Approx.mup'.by.mup} follows from \eqref{e:subsetRelation}, while the second term in \eqref{e:.mup.prob.measure} is obtained by using the fact that $\mun^g(c - \xi_n, c+ \xi_n] \leq 1$ which itself holds since $\mun^g$ is a probability measure.

Now, Proposition~\ref{prop:conv on Intervals} applies to $\mun^g.$ This, along with the given condition on $\|\epsilon_n\|_\infty$ and the fact that $\xi$ is arbitrary, then shows that $|
\mun'(c, \infty) - \mun^g(c, \infty)| \to 0$ in $\Lq,$ as desired. 
\end{proof}

\subsection{Weighted $Y(n, p)$ complex}
\label{subsec:Linial.Meshulam.Complex} 
Consider the weighted simplicial complex $(Y, \wg),$ where $Y$ is a sample of $Y(n, p),$ and $\wg$ is such that $\{\wg(\sigma): \sigma \in \cF^d(Y)\}$ is a set of i.i.d. random variables with some generic distribution $F,$ while $\wg(\sigma) = -\infty$ whenever $|\sigma| \leq d.$ This complex differs from $(\cKn, \wg)$ since not all the potential $d$-faces may be present here. Despite this, we now show that a version of our result holds for this complex as well. 
Let $M^g_{n, p}$ be a $d$-MSA in $(Y, \wu)$ and $D^g_{n, p}$ the set of death times in the associated  persistence diagram, whenever $\beta_{d - 1}(Y) = 0$ (so that $M^g_{n, p}$ exists).  Further, let
\[
    \mup^g = \frac{1}{\binom{n - 1}{d}} \sum_{\sigma \in M^g_{n, p}} \delta_{n p F(\wg(\sigma))} \quad \text{ and } \quad \tmup^g := \frac{1}{\binom{n - 1}{d}} \sum_{\Delta \in D^g_{n, p}} \delta_{n p F(\Delta)},
\]
whenever $\beta_{d - 1}(Y) = 0,$ and some arbitrary probability measure otherwise. 


\begin{corollary}
Suppose $p = (d \log n + \omega_n)/n$ for any function $\omega_n$ that tends to infinity. Then, Thorem~\ref{thm:bulk} holds for $\mup^g$ and $\tmup^g.$
\end{corollary}
\begin{proof}
As in the proof of Corollary~\ref{cor:Noisy.perturbations}, it suffices to show that an analog of Proposition~\ref{prop:conv on Intervals} holds for $\mup^g.$

From $(Y, \wg),$ we construct a coupled weighted complex $(\cKn, w_Y),$ where $\cKn \supseteq Y$ is the complete $d$-skeleton on $n$ vertices (as before), and $w_Y: \cKn \to [0, 1]$ is the weight function given by 
\begin{equation}
\label{d:w_Y}
    w_Y(\sigma) = 
    \begin{cases}
    U[p, 1] & \text{if $\sigma \in \cF^d(\cKn)\setminus \cF^d(Y),$} \\
    p F(w_g(\sigma)) & \text{if $\sigma \in \cF^d(Y)$ or $|\sigma| \leq d.$} 
    \end{cases}
\end{equation}

In the above expression, $U[p, 1]$ is an independent uniform random variable on the interval $[p, 1].$ It is easy to see that $\{w_Y(\sigma): \sigma \in \cF^d(\cKn)\}$ is a set of i.i.d. $U[0, 1]$ random variables. Also, $w_Y(\sigma) = 0$ whenever $|\sigma| \leq d.$ Therefore, Theorem~\ref{thm:bulk} readily applies to $(\cKn, w_Y).$

Let $\mun$ be as in \eqref{d:empirical.measure}, but defined in the context of $(\cKn, w_Y).$ Similarly, let $\Mn$ denote a $d$-MSA in $(\cKn, w_Y).$ Note that a $d$-MSA always exists in this complex, since $\cKn$ is a complete $d$-skeleton. 

Now, suppose the event $\{\beta_{d - 1}(Y) = 0\}$ holds. Then, from \cite[Lemma~23]{skraba2017randomly},  a $d$-MSA $M^g_{n, p}$ exists in $(Y, \wg).$ Further, from \eqref{d:w_Y}, the $d$-faces of $\cKn,$ that are present in $Y,$ have weights smaller than those that aren't; the former have weights less than or equal to $p,$ while the latter at least $p.$ Combining this with the fact a distribution function is always monotone, it follows that $M^g_{n, p}$ is also a $d$-MSA in $(\cKn, w_Y).$ 

Therefore,
\begin{align*}
    |\mup^g(c, \infty) - \mun(c, \infty)| \leq {} & |\mup^g(c, \infty) - \mun(c, \infty)| \indc[\beta_{d - 1}(Y) = 0] +  |\mup^g(c, \infty) - \mun(c, \infty)| \indc[\beta_{d - 1}(Y) \neq 0] \\    
    \leq {} & |\mup^g(c, \infty) - \mun(c, \infty)| \indc[\beta_{d - 1}(Y) \neq 0] \\
    \leq {} & 2 \indc[\beta_{d - 1}(Y) \neq 0],
\end{align*}
where the second relation makes use of the fact that $\mup^g = \mun$ on the event $\{\beta_{d - 1}(Y) = 0\},$ while the last holds since $\mup^g$ and $\mun$ are probability measures.

Now, due to the condition on $p,$ it follows from \cite[Theorem~1.1]{linial2006homological} and \cite[Theorem~1.1]{meshulam2009homological} that  $\lim_{n \to \infty} \Pr\{\beta_{d- 1}(Y) = 0\} = 1.$ The desired result is now easy to see. 
\end{proof}

Figure~\ref{fig:Ordinary.and.Scaled.Histograms} illustrates the above result for the case where $F$ is the uniform distribution on $[0, 1]$ and $p$ is constant. We consider two different cases for the $(n, d)$ pair. In each case, we also consider three different values of $p,$
and look at the $d$-MSA $M_{n, p}$ in one sample each of the resultant $Y(n, p)$ complex; we resample if the MSA does not exist. Notice that all the panels have two distinct plots: the blue one is the histogram corresponding to $\{n w(\sigma): \sigma \in M_{n, p}\},$ while the yellow one corresponds to $\{n p w(\sigma) : \sigma \in M_{n, p}\}.$ Clearly, unlike the blue plots, the yellow ones look similar irrespective of the $p$ values. 


\begin{figure}[ht!]
\begin{subfigure}{.495\textwidth}
  \centering
  \includegraphics[width=\linewidth]{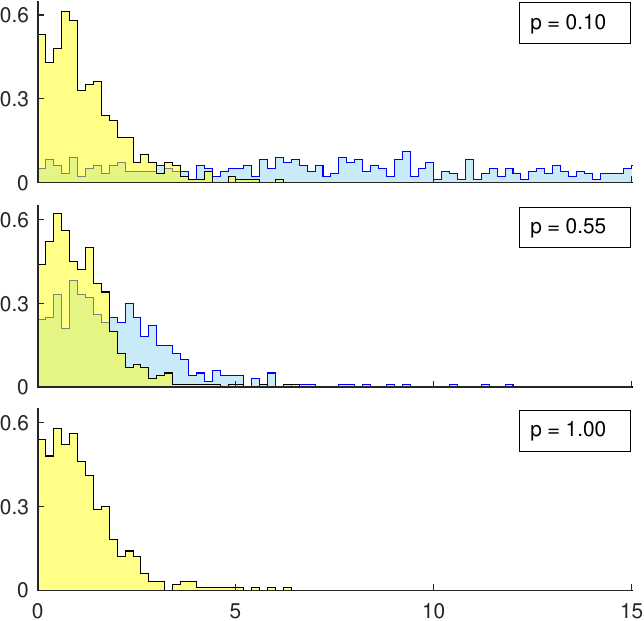}
  \caption{$n = 500, \quad d = 1$}
\end{subfigure}
\begin{subfigure}{.495\textwidth}
  \centering
  \includegraphics[width=\linewidth]{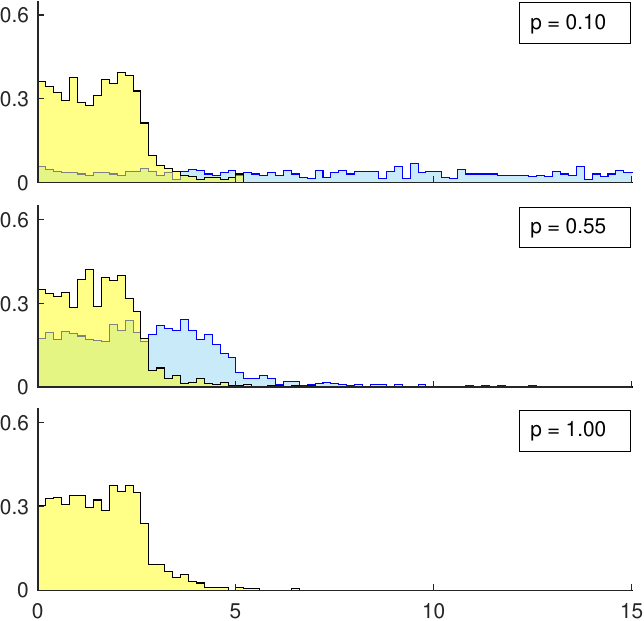}
  \caption{$n = 50, \quad d = 2$}
\end{subfigure}
\caption{Normalized histograms of $\{n p w(\sigma): \sigma \in M_{n, p}\}$ and $\{n w(\sigma): \sigma \in M_{n, p}\},$ depicted in yellow and blue, respectively. In the last panel, the two histograms coincide since $p = 1.$
\label{fig:Ordinary.and.Scaled.Histograms}}
\end{figure}


    

\section{Summary and Future Directions}
Our work went beyond the sum and studied the distribution of the individual weights in a random $d$-MSA. A key contribution here is an  explicit connection between the $d$-MSA weights and the $d$-shadow.

In future, one of the scenarios that we interested in exploring is the streaming setup. Here, the  $d$-faces along with their weights are revealed one by one and the $d$-MSA is then  incrementally updated by either including the face and then updating the $d$-MSA estimate or by discarding the face altogether. Note that we don't assume the faces are revealed in any order. The goal then is to understand how the distribution of the weights in the the $d$-MSA changes with time. 

Our preliminary conjecture on how the sum of these weights will change in the uniformly weighted setup is given in Figure~\ref{fig:Online.MSA}. In this figure, we consider two scenarios: i.) $n = 100$ and $d = 1$ and ii.) $n = 30$ and $d = 2.$ In both the scenarios, the weights of the $d$-faces are i.i.d.\ uniform $[0, 1]$ random variables which are revealed one at a time in an arbitrary order. That is, the weight of each $d$-face is initially presumed to be $1$ which then gets replaced with the actual value when it gets revealed. Accordingly, we begin with an arbitrary $d$-MSA and then sequentially update it as the weight of a new face gets revealed.

Let $M_n(k)$ denote the updated $d$-MSA after the weight of the $k$-th face is revealed. Also, let $c_1 = n/\tbinom{n}{d -1}$ and $c_2 =   \binom{n}{d + 1} \mu f$ for $f(x) = x.$ 
The blue and red curves in the figure show the value of $c_1 w(M_n(k))$ and $c_2/k,$ respectively. The  red curve is our guess for $c_1 M_n(k)$ based on the results in this paper. Finally, the black curve shows the constant value $\mu x.$  Clearly, the blue and red curves closely match and they both get close to the black one as the value of $k$ increases.

\begin{figure}[ht]
\begin{subfigure}{.495\textwidth}
  \centering
  \includegraphics[width=\linewidth]{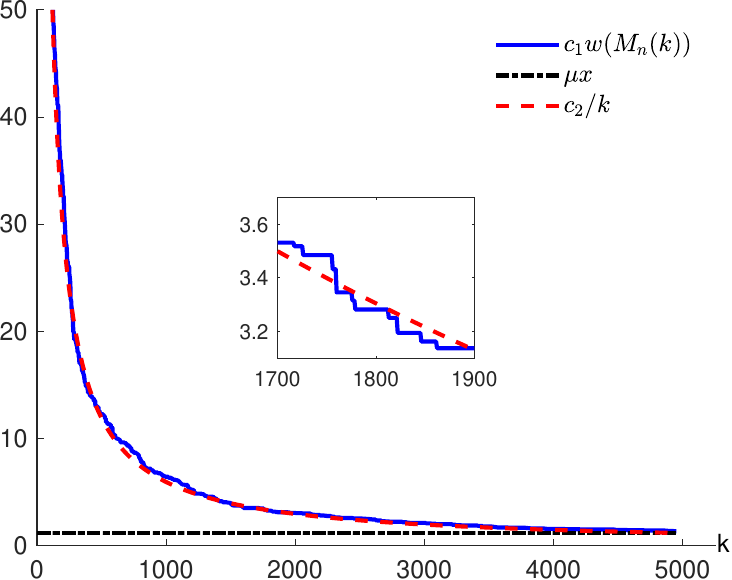}
  \caption{$n = 100, d = 1$}
  \label{fig:stream-first}
\end{subfigure}
\begin{subfigure}{.495\textwidth}
  \centering
  \includegraphics[width=\linewidth]{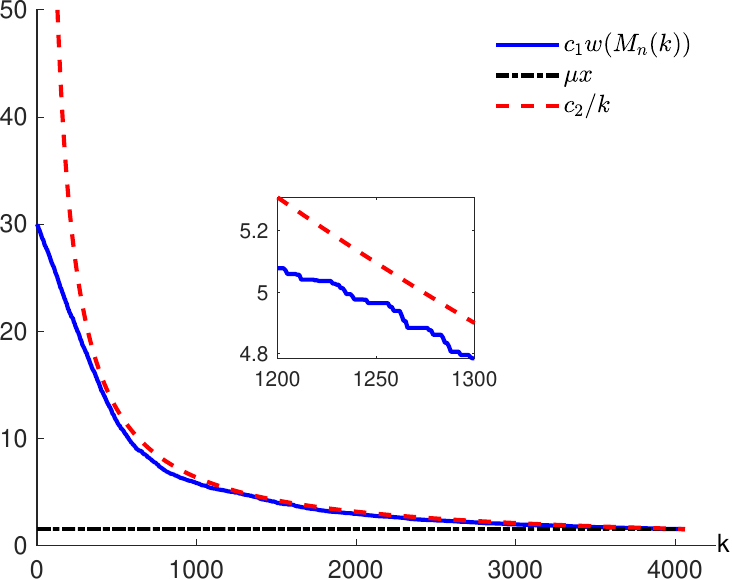}
  \caption{$n = 30, d = 2$}
  \label{fig:stream-second}
\end{subfigure}
\caption{Comparison of the scaled weight (blue) of the actual $d$-MSA with our conjecture (red). The values of $c_1$ and $c_2$ are $n/\binom{n - 1}{d}$ and $ \binom{n}{d + 1} \mu x,$ respectively. The black curve denotes the value of $\mu x.$ Recall that $\mu x$ is $\zeta(3) \approx 1.2$ for $d=1$ and $1.56$ for $d=2$.   \label{fig:Online.MSA}}
\end{figure}

Some of the other directions that we wish to pursue in the future are as follows. 
\begin{enumerate}
    \item  \emph{Geometric random graphs and complexes}: Extending our results from the \Erdos-\Renyi\ and Linial-Meshulam models to geometric random graphs and geometric complexes and understanding differences between the two scenarios.
    
    \item \emph{Large deviation results and central limit theorems}: Obtaining general large deviation results for the Kolmogorov  distance between $\mun$ and $\mu$ as well as central limit theorems characterizing the variance of the weight distribution.
    
    \item \emph{Rates of convergence}: Deriving the rate at which $\mun$ converges to $\mu.$ 

\end{enumerate}

\section*{Acknowledgments} 
The authors would like to thank Matthew Kahle, Omer Bobrowski, Primoz Skraba, Ron Rosenthal, Robert Adler, Christina Goldschmidt, D. Yogeshwaran, and Takashi Owada for useful suggestions. A portion of this work was done when Gugan Thoppe was a postdoc with  Sayan Mukherjee at Duke University. 

\bibliographystyle{amsplain}


\begin{dajauthors}
\begin{authorinfo}[nico]
  Nicolas Fraiman\\
  Assistant Professor\\
  University of North Carolina at Chapel Hill\\
  Chapel Hill, NC, United States\\
  fraiman\imageat{}email\imagedot{}unc\imagedot{}edu \\
  \url{https://fraiman.web.unc.edu/}
\end{authorinfo}
\begin{authorinfo}[sayan]
  Sayan Mukherjee\\
  Professor\\
  Duke University, \\
  Durham, NC, United States\\
  Center for Scalable Data Analytics and Artificial Intelligence, Universit\"at Leipzig \\
Max Planck Institute for Mathematics in the Sciences\\
Leipzig, Saxony, Germany
  sayan.mukherjee\imageat{}mis\imagedot{}mpg\imagedot{}de \\
  \url{https://sayanmuk.github.io/}
\end{authorinfo}
\begin{authorinfo}[gugan]
  Gugan Thoppe\\
  Assistant Professor\\
  Indian Institute of Science\\
  Bengaluru, India\\
  gthoppe\imageat{}iisc\imagedot{}ac\imagedot{}in\\
  \url{https://sites.google.com/site/gugancth/}
\end{authorinfo}
\end{dajauthors}

\end{document}